\documentclass{amsart}
\usepackage{amsmath,amssymb,amsthm,amsfonts,amscd}
\usepackage{microtype}
\usepackage{mathscinet}
\usepackage[pdftex]{color}
\usepackage[bookmarks=true,hyperindex,pdftex,colorlinks, citecolor=MidnightBlue,linkcolor=MidnightBlue, urlcolor=MidnightBlue]{hyperref}
\usepackage[dvipsnames]{xcolor}

\newtheorem{theorem}{Theorem}
\newtheorem{proposition}[theorem]{Proposition}

\theoremstyle{definition}

\theoremstyle{remark}

\numberwithin{equation}{section}

\newcommand{\norm}[1]{\left\Vert#1\right\Vert}
\newcommand{\clco}{\mathop{\overline{\mathrm{conv}}}\nolimits}

\newcommand{\Free}{{\mathcal F}}
\newcommand{\Lip}{{\mathrm{Lip}}_0}

\newcommand{\Real}{\mathbb{R}}
\newcommand{\set}[1]{\left\{#1\right\}}

\newcommand\N{\mathbb{ N}}
\newcommand\Pe{\mathbb{ P}}

\begin{document}

\title[Pe\l czy\'nski space is isomorphic to the free space over a compact]{Pe\l czy\'nski space is isomorphic to the Lipschitz free space over a compact set}


\author[L. C. Garc\'ia-Lirola]{Luis C. Garc\'ia-Lirola}
\address{Department of Mathematical Sciences, Kent State University, Kent OH 44242, USA}
\email{lgarcial@kent.edu}
\thanks{The first author was partially supported by the grants MINECO/FEDER MTM2017-83262-C2-2-P, Fundaci\'on S\'eneca CARM
	19368/PI/14 and by a postdoctoral grant in the framework of \emph{Programa
		Regional de Talento Investigador y su Empleabilidad} from \emph{Fundaci\'on S\'eneca - Agencia de Ciencia y Tecnolog\'ia de la Regi\'on de Murcia}.}

\author[A. Proch\'azka]{Anton\'in Proch\'azka}
\address{Universit\'e Bourgogne Franche-Comt\'e, Laboratoire de Math\'ematiques UMR 6623, 16 route de Gray,
	25030 Besan\c con Cedex, France}
\email{antonin.prochazka@univ-fcomte.fr}
\thanks{The second author was supported by the French ``Investissements d'Avenir'' program, project ISITE-BFC (contract ANR-15-IDEX-03).}



\date{August, 2018}

\commby{}

\begin{abstract}
	We prove the result stated in the title.
	This provides a first example of an infinite-dimensional Banach space whose Lipschitz free space is isomorphic to the free space of a compact set.
\end{abstract}

\maketitle


Let us first motivate the result stated in the title. 
For a metric space $M$ with a distinguished point $0 \in M$ (commonly called a \emph{pointed metric space}), the \emph{Lipschitz free space} $\Free(M)$ is the norm-closed linear span of the evaluation functionals, i.e. $\set{\delta(x): x \in M}$, in the space $\Lip(M)^*$. 
Here the Banach space $\Lip(M)=\set{f \in \Real^M: f \mbox{ Lipschitz}, f(0)=0}$ is equipped with the norm \[\displaystyle\norm{f}_L:=\sup\set{\frac{f(x)-f(y)}{d(x,y)}:x \neq y}.\]
One of the main features of the free spaces is that every Lipschitz map $f:M \to N$ which fixes the zero induces a linear map $\hat{f}:\Free(M)\to \Free(N)$ such that $\delta_N \circ f=\hat{f}\circ \delta_M$ and $\norm{\hat{f}}=\norm{f}_L$.
For a quick proof and some other basic properties we refer the reader to the paper \cite{CDW16}. 

The operation of constructing a Lipschitz free space over a given metric space is thus a functor from the category of pointed metric spaces (whose morphisms are the Lipschitz maps that fix zero) into the category of Banach spaces whose morphisms are the bounded linear maps. 
This functor is not injective though. 
Some information about the metric space is lost on the way to its free space.
Indeed, among the simplest examples we can mention the metric spaces $[0,1]$ and $\Real$ whose free spaces are (isometrically) isomorphic since $\Free([0,1])=L^1[0,1]$ and $\Free(\Real)=L^1(\Real)$. 
More generally, Kaufmann \cite{Ka15} has proved that for every Banach space $X$, the space $\Free(X)$ and $\Free(B_X)$ are isomorphic. 
It can also happen that $\Free(X)$ is isomorphic to $\Free(Y)$ when $X$ and $Y$ are Lipschitz non-isomorphic separable Banach spaces. 
This is the case for example for $X=c_0$ and $Y=C[0,1]$, as was shown by Dutrieux and Ferenczi~\cite{DutFer}.
On the other hand, if $X$ has the bounded approximation property (BAP for short, see the definition below) while $Y$ does not, then $\Free(X)$ and $\Free(Y)$ cannot be isomorphic. 
Indeed, a celebrated result of Godefroy and Kalton~\cite{GK03} says that $X$ enjoys the BAP if and only if $\Free(X)$ does. 

An intriguing open problem is whether $\dim X \neq \dim Y$ implies that $\Free(X)$ is not isomorphic to $\Free(Y)$. 
The only partial solution is due to Naor and Schechtman~\cite{NaorSchechtman} who have shown that $\Free(\Real^2)$ does not embed into $\Free(\Real)$.
A fortiori, $\Free(\Real^2) \not\simeq \Free(\Real)$.
But otherwise the fog is so dense in these parts that it is not even clear whether $\dim X=\infty$ and $\dim Y<\infty$ implies $\Free(X) \not\simeq \Free(Y)$. 
Looking again at the above mentioned result of Kaufmann, this problem would be settled if one could prove that $\Free(X)$ is not isomorphic to the free space over any compact metric space when $\dim X=\infty$. 
That is certainly the case for any non-separable Banach space $X$ (since the density characters of $X$ and $\Free(X)$ are equal for any infinite $X$) but it turns out to be false in general. 
Indeed, it is enough to combine our result with the well known fact that $\Pe$ and $\Free(\Pe)$ are isomorphic (see~\cite[Remark on page 139]{GK03}).

\begin{theorem}\label{t:main}
	There exists a compact convex subset $K$ of the Pe\l czy\'nski space $\Pe$ such that $\Pe$ is isomorphic to $\Free(K)$. 
\end{theorem}

Recall that, up to isomorphism, $\Pe$ is the unique separable  Banach space with a Schauder basis and such that every separable Banach space with the BAP is isomorphic to a complemented subspace of $\Pe$. 
We refer the reader to \cite{AlbiacKalton} for the construction of this space which shows that every Banach space with a basis is complemented in $\Pe$. 
The stronger property that every Banach space with the BAP is complemented in $\Pe$ is proved in \cite{Kadec,Pel2}.

A Banach space $X$ is said to have the \emph{bounded approximation property} if there exists $\lambda>0$ such that for every $x_1,\ldots,x_n \in X$ and every $\varepsilon>0$ there exists a finite rank linear operator $T:X\to X$ which satisfies $\norm{T}\leq \lambda$ and $\norm{x_i-Tx_i}\leq \varepsilon$ for all $1\leq i\leq n$.
If the above is satisfied with a given $\lambda>0$, we say that $X$ has the $\lambda$-BAP.

In the proof of Theorem~\ref{t:main} we will construct a compact convex $K\subset \Pe$ in such a way that $\Free(K)$ has the BAP. 
In fact, $K$ will satisfy the hypothesis of the following criterion which is an easy corollary of a much deeper result of Perneck\'a and Smith~\cite{PS15}.

\begin{proposition}\label{compactBAP} Let $X$ be a Banach space and $K\subset X$ be a 
	closed
	convex subset containing $0$. 
	Assume that there exist $\lambda\geq 1$ and a sequence $(T_n)_n$ of finite-rank operators on $X$ such that $\norm{T_n}\leq \lambda$ and $T_n(K)\subset K$ for each $n$, and $(T_n)_n$ converges pointwise to the identity on $K$. 
	Then $\Free(K)$ has the $\lambda$-BAP. 
\end{proposition}
The proof is a slight modification of a part of \cite[Theorem 5.3]{GK03}.
\begin{proof} 
	By a density argument, it suffices to show that if $\varepsilon>0$ and $x_1,\ldots x_k\in K$ then there exists a finite rank linear map $T:\Free(K)\to\Free(K)$ such that $\norm{T}\leq\lambda$ and $\norm{T\delta(x_i)-\delta(x_i)}<\varepsilon$ for $1\leq i\leq k$. 
	Fix $n$ such that $\norm{T_n(x_i)-x_i}<\varepsilon/2$ for $1\leq i\leq k$. 
	Let $\widehat{T_n}\colon \Free(K)\to\Free(T_n(K))$ be the induced linear map. 
	Since 
	$\overline{T_n(K)}$ 
	is a finite-dimensional compact convex set, 
	$\Free(T_n(K))=\Free(\overline{T_n(K)})$ 
	has the 1-BAP~\cite{PS15}. 
	Thus we can find a finite-rank operator $S\colon\Free(T_n(K))\to\Free(T_n(K))$ so that $\norm{S}= 1$ and $\norm{S\widehat{T}_n\delta(x_i)-\widehat{T}_n(\delta(x_i))}<\varepsilon/2$ for $1\leq 1\leq k$. 
	Then $T=S\widehat{T}_n$ does the work. 
\end{proof}

\begin{proof}[Proof of Theorem~\ref{t:main}] Let $\{e_n : n\in\N\}$ be a normalized Schauder basis of $\Pe$ with the associated projections $(P_n)_n$. 
	Consider $K=\clco\{e_n / n: n\in \mathbb N\}$, which clearly is a compact convex set satisfying $\Pe = \overline{\operatorname{span}}(K)$. 
	By \cite[Lemma~2.1]{DL08} or \cite[Theorem~4]{GO14} 
	we get that $\Pe$ 
	is isomorphic to a complemented subspace of $\Free(K)$. 
	Notice that $P_n(K)\subset K$ for each $n$. 
	Thus, the space $\Free(K)$ has the BAP by Proposition~\ref{compactBAP}. 
	Now, the universal property of $\Pe$ yields that $\Free(K)$ is isomorphic to a complemented subspace of $\Pe$. 
	Since $\Pe$ is isomorphic to its $\ell_1$-sum (see~\cite[proof of Theorem~13.3.1]{AlbiacKalton}),  the conclusion follows by applying the standard Pe\l czy\'nski's decomposition method.
\end{proof}

With our appetite whetted by this positive result we ask: is it true that for every separable infinite dimensional Banach space $X$ there exists a compact metric space $K$ such that $\Free(X) \simeq \Free(K)$?
Can such $K$ always be found inside $X$?
Let us remark that $\Free(X)$ is isomorphic to a complemented subspace of $\Free(K)$ whenever $K$ is a generating convex subset of $\Free(X)$.
Indeed, it is the contents of \cite[Lemma 2.1]{DL08} or \cite[Theorem~4]{GO14}.

Notice that the proof of Theorem~\ref{t:main} could also be finished by proving that $K$ is a Lipschitz retract of $\Pe$ (instead of proving that $\Free(K)$ has the BAP).
Nevertheless we don't know if this is the case. 
More generally, the question of whether every separable Banach space admits a compact convex generating Lipschitz retract was raised in~\cite[Question 3]{GO14}.
Possible consequences of the affirmative answer to this question are discussed in~\cite[Problem 6.3]{GodefroySurvey}.

\end{document}